\documentclass[11pt,reqno]{amsart}
\usepackage{amsmath,amssymb,amsthm, epsfig}
\usepackage{hyperref}
\usepackage[mathscr]{euscript}

\usepackage{comment}

\usepackage{xcolor}
\usepackage[nameinlink]{cleveref}
\hypersetup{
	colorlinks,
	linkcolor={blue!50!black},
	citecolor={blue!50!black},
	urlcolor={blue!50!black}
}

\usepackage{enumerate}
\usepackage{mathrsfs}
\usepackage{esint}

\newtheorem{theorem}{Theorem}
\newtheorem{definition}{Definition}
\newtheorem{lemma}{Lemma}
\newtheorem{proposition}{Proposition}
\newtheorem{corollary}{Corollary}
\newtheorem{remark}{Remark}

\date{}
\numberwithin{equation}{section}
\numberwithin{theorem}{section}
\numberwithin{lemma}{section}
\numberwithin{corollary}{section}
\numberwithin{remark}{section} 
\numberwithin{proposition}{section}
\numberwithin{definition}{section}

\newcommand{\Div}{\operatorname{div}}

\newcommand{\R}{\mathbb{R}}
\newcommand{\supp}{\operatorname{supp } }
\newcommand{\diam}{\operatorname{diam}}

\newcommand{\loc}{\operatorname{loc}}

\def\XXint#1#2#3{{\setbox0=\hbox{$#1{#2#3}{\int}$ }
		\vcenter{\hbox{$#2#3$ }}\kern-.6\wd0}}

\begin{document}

\subjclass[2020]{35B65, 35J70, 35J92}

\keywords{Regularity, Degenerate elliptic PDEs, Morrey spaces.}


\title[Higher regularity for degenerate PDEs in Morrey spaces]{Higher H\"older regularity for degenerate elliptic PDEs with data in Morrey spaces}

\author[Giuseppe Di Fazio]{Giuseppe Di Fazio}
\address{Dipartimento di Matematica e Informatica, University of Catania, Italy}{}
\email{giuseppe.difazio@unict.it}

\author[R. Teymurazyan]{Rafayel Teymurazyan}
\address{Applied Mathematics and Computational Sciences (AMCS), Compu\-ter, Electrical and Mathematical Sciences and Engineering Division (CEMSE), King Abdullah University of Science and Technology (KAUST), Thuwal, 23955 -6900, Kingdom of Saudi Arabia}{}
\email{rafayel.teymurazyan@kaust.edu.sa} 

\author[J.M.~Urbano]{Jos\'{e} Miguel Urbano}
\address{Applied Mathematics and Computational Sciences (AMCS), Compu\-ter, Electrical and Mathematical Sciences and Engineering Division (CEMSE), King Abdullah University of Science and Technology (KAUST), Thuwal, 23955 -6900, Kingdom of Saudi Arabia and CMUC, Department of Mathematics, University of Coimbra, 3000-143 Coimbra, Portugal}{} 
\email{miguel.urbano@kaust.edu.sa} 

\begin{abstract}
We establish sharp local $C^{1,\alpha}-$regularity for weak solutions to degenerate elliptic equations of $p-$Laplacian type with data in Morrey spaces. The proof relies on the Fefferman--Phong inequality and standard tools from regularity theory for nonlinear PDEs.
\end{abstract}

\date{\today}
	
\maketitle
	
\section{Introduction}

This note addresses the issue of higher regularity for degenerate elliptic PDEs of the form
\begin{equation}\label{general equation}
    -\Div A(x,u, Du)=B(x,u,Du)
\end{equation}
in a bounded domain $\Omega\subset\R^n$. The functions $A$ and $B$ satisfy suitable structural assumptions modeled on the $p-$Poisson equation
\begin{equation}\label{mainequation}
	-\Delta_p u = f,
\end{equation}
where $\Delta_p u:=\operatorname{div}(|Du|^{p-2} Du)$ is the $p-$Laplacian operator. The novelty is that $f$ enjoys a modest degree of integrability, namely it is merely an $L^1-$function satisfying a decay condition on its integrals over balls, \textit{i.e.}, it belongs to a suitable Morrey space. More precisely, we will assume 
\begin{equation}\label{f is in Morrey}
    f\in L^{1,\lambda}(\Omega), \qquad n-1<\lambda<n.
\end{equation}
The choice of a Morrey space is motivated by the fact that, in some special cases, this condition is also necessary for regularity (see \cite{DFnec} and \cite{Rako}).

The study of regularity for equations of this type has a rich and long history. After the pioneering De Giorgi--Nash--Moser theory, the work of Ladyzhenskaya and Uraltseva (see \cite{LU68}, and also \cite{E82}) established the local $C^{1,\alpha}-$regularity of weak solutions for the homogeneous equation, and was later extended to equations with a quite general right-hand side satisfying suitable integrability assumptions (see, for example, \cite{DiB83, T84}). The study of regularity under less restrictive integrability assumptions was pioneered by Morrey in \cite{M66} and advanced through contributions from various authors (see \cite{DiF92}, \cite{M24} and references therein). Under assumption \eqref{f is in Morrey}, weak solutions of \eqref{mainequation} are known to be locally of class $C^{0,\alpha}$, for some $\alpha\in(0,1)$ (see \cite[Theorem 2.5]{DiF92}, \cite[Theorem 3]{Z94}, \cite{DFZ2022}, \cite{DFZ-DiBe} and \cite{DFnec}). Moreover, they are differentiable under several different assumptions on $f$ (see, \textit{e.g.}, \cite{M24} and \cite{BH10, DM11}). 

Our main result establishes higher H\"older regularity, namely the $C^{1,\alpha}$ local regularity, for
\begin{equation*}
\alpha=
\begin{cases}
\min\left(\gamma,\lambda + 1 -\frac{2n}{p}\right),\quad \frac{2n}{\lambda+1}<p \leq 2;\\
\\
\min\left(\gamma,\dfrac{\lambda +1 -n}{p-1}\right),\quad 2\le p<n.
\end{cases}
\end{equation*}
Here, $\gamma\in(0,1)$ is the H\"older regularity exponent for the gradient of the solution to the homogeneous equation having $u$ as boundary data. Such a result is known in the linear case $p=2$. It is shown in \cite[Theorem 2.8]{DiF92} that if $u$ is a solution to $Lu=f$, where $L$ is a linear uniformly elliptic operator with smooth coefficients, then $u\in C^{1,\alpha}_{\loc}(\Omega)$ for some $\alpha\in(0,1)$, provided \eqref{f is in Morrey} holds. Still, the proof heavily relies on properties of the Green function that are not available for the nonlinear degenerate case. 

Our analysis relies heavily on the powerful Fefferman--Phong inequality (see \Cref{FP inequality}), a weighted Poincar\'e-type inequality instrumental in estimating the term containing the right-hand side. This inequality dates back to the 1980s \cite{F83, FP81}, and it was used to understand the distribution of eigenvalues of self-adjoint operators. A simplified proof of the Fefferman--Phong inequality from \cite{F83} was given in \cite{CF90} under an $L^r$ assumption on the right-hand side. For other variants of the Fefferman--Phong inequality, see \cite{GiuPie}, treating the case of general metric spaces with a doubling measure, and \cite{K24}, containing the extension to the fractional setting. 

The note is structured as follows. \Cref{review_section} overviews the main properties of the function spaces used in our analysis. To ensure clarity, in \Cref{results_section}, we write a detailed proof of the main result for the model equation \eqref{mainequation}. In \Cref{generalization}, we sketch the proof for the general case of \eqref{general equation}. 

\section{Morrey spaces and Stummel classes} \label{review_section}

This section provides a brief overview of Morrey spaces and Stummel classes. We refer the reader to \cite{PKJF13, SDH1, SDH2} for a detailed and comprehensive treatment. Throughout the note, we assume that $\Omega\subset\R^n$ is a bounded domain with a Lipschitz boundary. 

\begin{definition}
Let $1\leq p<\infty$ and $\lambda \geq 0$. The Morrey space $L^{p,\lambda}(\Omega)$ is the set of functions $f\in L^p(\Omega)$ for which 
\begin{equation} \label{Morrey_norm}
\|f\|_{L^{p,\lambda}(\Omega)}:=\sup_{\substack{x\in\Omega\\0<r<\diam(\Omega)}} \left(\frac{1}{r^\lambda}\int_{\Omega\cap B_r(x)}|f(y)|^p\,dy\right)^{1/p}<\infty.
\end{equation}
The space $L^{p,\lambda}(\Omega)$ is a Banach space (a Hilbert space if $p=2$) under the norm \eqref{Morrey_norm}.
\end{definition}

\begin{remark}
For any $1 \leq p<\infty$, $L^{p,0}(\Omega)=L^p(\Omega)$ and $L^{p,n}(\Omega)=L^\infty(\Omega)$. 
Furthermore, if $\lambda>n$, then $L^{p,\lambda}(\Omega)=\{0\}$ (see \cite[Theorem 5.5.1]{PKJF13}).
\end{remark}

Morrey spaces enjoy several embedding properties; for the following result, see \cite[Theorem 5.5.1]{PKJF13}.
\begin{proposition}\label{inc}
If $1\le p\le q<\infty$ and $\lambda,\mu\in[0,n]$ are such that 
$$
\dfrac{n-\mu}{q} \leq \dfrac{n-\lambda}{p},
$$
then
$$
L^{q,\mu} (\Omega) \subseteq  L^{p,\lambda} (\Omega).
$$
Moreover, there exists a constant $c>0$ such that
$$
\|f\|_{L^{p,\lambda}(\Omega)} \leq  c \, \|f\|_{L^{q,\mu}(\Omega)}.
$$
\end{proposition}

Next, we recall the definition of Stummel-Kato classes (see, for example, \cite{as, K73, S56}), which are strictly related to Morrey spaces. 

\begin{definition}
Let $f\in L^1(\Omega)$, and $1\leq p < n$. Set 
\begin{equation} \label{stummel-kato-modulus}
\eta(r)=\eta_p(f,r):=\sup_{x\in\Omega} \int_{\Omega\cap B_r(x)}\dfrac{|f(y)|}{|x-y|^{n-p}}\,dy\,.
\end{equation}
We say that $f$ belongs to the Stummel-Kato class $S_p(\Omega)$ if the Stummel-Kato modulus \eqref{stummel-kato-modulus} is finite and, moreover,
$$
\lim_{r\downarrow 0}\eta(r)=0.
$$
In case $\eta(r)$ is just bounded around zero, we say that $f\in\Tilde{S}_p(\Omega)$. 
\end{definition}

For the proof of the following result, we refer to \cite{SDH1} or \cite{DiF92}.

\begin{lemma}\label{inclusions}
If $1\le p<n$ and $n-p<\lambda<n$, then 
$$
L^{1,\lambda}(\Omega)
\subset  S_p(\Omega) \subset\Tilde{S}_p(\Omega)\subset L^{1,n-p}(\Omega).
$$
Moreover
$$
\eta(r)\le c\,r^{\lambda-n+p}\|f\|_{L^{1,\lambda}(\Omega)},
$$
for a constant $c>0$ independent of $f$ and $r$.
\end{lemma}
We also recall the Fefferman--Phong inequality, a weighted Poincar\'e type inequality fundamental in our analysis (see, \textit{e.g.}, \cite[Corollary 3.4]{GiuPie} and \cite{DFZ-SUM-FP}). 

\begin{theorem}\label{FP inequality}
If $1\le p<n$ and $f\in\Tilde{S}_p(\Omega)$, then there exists a constant $c>0$, depending only on $p$ and $n$, such that, for every ball $B\subset\Omega$ with radius $r>0$, one has
$$
\int_{B}|f||\varphi|^p\,dx 
\le 
c\,\eta(2r)\int_{B}|D\varphi|^p \,dx,
$$
for all $\varphi\in W^{1,p}_0(B)$.
\end{theorem}

As a consequence of \Cref{FP inequality} and \Cref{inclusions}, we obtain the following corollary, which plays a crucial role in proving our main result.

\begin{corollary}
If $1\le p< n$, $n-p<\lambda<n$, and $f\in L^{1,\lambda}(\Omega)$, then there exists a constant $c>0$, depending only on $p$, $\lambda$ and $n$, such that for every ball $B\subset\Omega$ with radius $r>0$, one has
\begin{equation}\label{constantestimate}
\int_{B}|f||\varphi|^p\,dx 
\le 
c \, r^{\lambda-n+p}\|f\|_{L^{1,\lambda}(\Omega)} \int_{B}|D\varphi|^p\,dx,
\end{equation}
for all $\varphi\in W^{1,p}_0(B)$.
\end{corollary}

We close this section by recalling the celebrated Campanato-Meyers characterization of H\"older spaces (see \cite{PKJF13}). We denote with $(u)_r$ the integral average of $u$ over a ball with radius $r$.

\begin{theorem}\label{campanato}
Let $1\le p<\infty$ and $0<\gamma\le1$.  
If there exists $c>0$ such that
$$\fint_{B}|u-(u)_r|^p\,dx \le c\, r^{\gamma p},$$
for every ball $B\subset\Omega$ with radius $r>0$, then $u\in C^{0,\gamma}_{\loc}(\Omega)$. 
\end{theorem}

\section{Higher H\"older regularity}\label{results_section}

In this section, we establish our main result. As noted earlier, for clarity, we first provide a detailed proof for the model equation \eqref{mainequation}, postponing the treatment of the general case \eqref{general equation} to \Cref{generalization}.

We start with the definition of a weak solution.

\begin{definition}\label{definition of solutions}
Let $1<p<n$, $n-1<\lambda<n$ and $f\in L^{1,\lambda}(\Omega)$. A function $u\in W^{1,p}(\Omega)$ is a weak solution to \eqref{mainequation} if
\begin{equation}\label{weaksolution}
\int_\Omega|Du|^{p-2}Du\cdot D\varphi\,dx=\int_\Omega f\varphi\,dx,
\quad \forall \varphi \in W^{1,p}_0(\Omega).
\end{equation}
\end{definition}

We first show that the right-hand side of \eqref{weaksolution} is finite. Let $B$ be a ball with radius $r>0$, contained in $\Omega$. For $\varphi\in C_0^\infty(B)$, the finiteness of the integral follows from the Fefferman--Phong inequality. Indeed, using H\"older's inequality, \eqref{Morrey_norm} and \eqref{constantestimate}, we obtain
\begin{eqnarray*}
\int_{B} |f| |\varphi|\,dx 
& 
\leq 
& 
\left( \int_{B} |f|\,dx \right)^{1/p'} \left( \int_{B} |f||\varphi|^p\,dx \right)^{1/p}\nonumber\\
& \leq & 
c\, \|f\|_{L^{1,\lambda}(B)}^{\frac{1}{p'}}r^{\frac{\lambda}{p'}}
\|f\|_{L^{1,\lambda}(B)}^{\frac{1}{p}}r^{\frac{\lambda-n+p}{p}}  \left(\int_B|D\varphi|^p\,dx \right)^{1/p}\nonumber\\
&\leq &c\,r^{\lambda+1-\frac{n}{p}}\|f\|_{L^{1,\lambda}(\Omega)}\|\varphi\|_{W^{1,p}(\Omega)}<\infty.
\end{eqnarray*}
Since $\lambda+1-\frac{n}{p}>0$ and $r\le\diam(\Omega)$, we get
\begin{equation}\label{important inequality 1}
    \int_{B} |f| |\varphi|\,dx\le c\,
    \|f\|_{L^{1,\lambda} (\Omega)}\|\varphi\|_{W^{1,p}(\Omega)}, \quad \forall\varphi\in C_0^\infty(B),
\end{equation}
where $c>0$ depends only on $p$, $n$, $\lambda$ and $\diam(\Omega)$. Next, we aim to extend \eqref{important inequality 1} from $B$ to $\Omega$. Indeed, if $\varphi\in C_0^\infty(\Omega)$, then we can cover the $\supp\varphi$ with finitely many balls $\{B_i\}_{i=1}^m\subset\Omega$, and construct a partition of unity $\{\eta_i\}_{i=1}^m\in C_0^\infty(B_i)$, with $\eta_i\in[0,1]$, such that, in $\supp\varphi$, one has
$$
\sum_{i=1}^m\eta_i=1.
$$
Observe that if $\varphi_i:=\varphi\eta_i\in C_0^\infty(B_i)$, $1\le i\le m$, then
\begin{equation}\label{important inequality 2}
    \|\varphi_i\|_{W^{1,p}(\Omega)}\le C_i\|\varphi\|_{W^{1,p}(\Omega)},
\end{equation}
where the constant $C_i>0$ depends only on $\eta_i$. Furthermore, in $\supp\varphi$, we have
$$
\varphi=\varphi\sum_{i=1}^m\eta_i=\sum_{i=1}^m\varphi_i.
$$
Combining the latter with \eqref{important inequality 1} and \eqref{important inequality 2}, we get
\begin{equation*}
    \begin{split}
        \int_\Omega|f||\varphi|\,dx&\le\sum_{i=1}^m\int_{B_i}|f||\varphi_i|\,dx\\
        &\le c\,
    \|f\|_{L^{1,\lambda} (\Omega)}\sum_{i=1}^m\|\varphi_i\|_{W^{1,p}(\Omega)}\\
    &\le C\,
    \|f\|_{L^{1,\lambda} (\Omega)}\|\varphi\|_{W^{1,p}(\Omega)},
    \end{split}
\end{equation*}
where $C>0$ is a constant depending only on $c$ and $\displaystyle\sum_{i=1}^mC_i$. Thus, 
\begin{equation}\label{compactsupport}
\int_{\Omega} |f| |\varphi|\,dx 
\le
c\,
\|f\|_{L^{1,\lambda}(\Omega)}\|\varphi\|_{W^{1,p}(\Omega)},\quad \forall\,\varphi\in C_0^\infty(\Omega).
\end{equation}
Observe now that \eqref{compactsupport} remains true for test functions in $W^{1,p}_0(\Omega)$. Indeed, if $\varphi\in W^{1,p}_0(\Omega)$ and $\varphi_i\in C_0^\infty(\Omega)$ is such that $\varphi_i\to\varphi$ in $W^{1,p}(\Omega)$, as $i\to\infty$, then, up to a subsequence (still denoted by $\varphi_i$), one has $\varphi_i\to\varphi$, a.e. in $\Omega$. By \eqref{compactsupport},
\begin{equation}\label{appoximationestimate}
    \int_\Omega |f| |\varphi_i|\,dx
    \le
    c\,
    \|f\|_{L^{1,\lambda}(\Omega)}\|\varphi_i\|_{W^{1,p}(\Omega)}.
\end{equation}
Since $\|\varphi_i\|_{W^{1,p}(\Omega)}\to\|\varphi\|_{W^{1,p}(\Omega)}$,
employing Fatou's lemma, \eqref{appoximationestimate} yields
\begin{equation*}
    \int_\Omega |f||\varphi|\,dx\le
    \liminf_{i\to\infty} \int_\Omega|f||\varphi_i|\,dx\le
    c\,
    \|f\|_{L^{1,\lambda}(\Omega)}\|\varphi\|_{W^{1,p}(\Omega)}<\infty.
\end{equation*}

\begin{remark}
We stress the low integrability assumption on $f$, which is merely an $L^1-$function. In fact, our results still hold if $f$ is replaced by a measure. A measure $\mu$ in $\Omega$ is said to belong to the Morrey space of measures
$ML^{1,\lambda}(\R^n)$ if there exists a constant $c$ such that
$$
\left| \mu(B_r(x)) \right| \leq c\, r^\lambda,
$$
for any $x \in \R^n$ and $r>0$.
\end{remark}
To proceed, we first revisit the following lemma from \cite{LU68} (see also \cite{DiB83, E82, T84}), which will be used in proving our main result. 

\begin{lemma}\label{regularity of v}
Let $1<p<\infty$ and $u\in W^{1,p}(\Omega)$. There exists $\gamma\in (0,1)$ such that the problem
\begin{equation*}
	\begin{cases}
		-\Div(|Dv|^{p-2}Dv)=0 \quad \textrm{in} \ \Omega\\
		v-u\in W^{1,p}_0(\Omega)
	\end{cases}
\end{equation*}
has a unique solution $v\in C^{1,\gamma}_{\loc}(\Omega)$.
\end{lemma}

We are now ready to prove our main result.

\begin{theorem}\label{main result} 
Let $n-1<\lambda<n$, $f\in L^{1,\lambda}(\Omega)$, and $\gamma\in(0,1)$ be as in \Cref{regularity of v}.  If $u$ is a weak solution of \eqref{mainequation}, then $u\in C_{\loc}^{1,\alpha}(\Omega)$, where
\begin{equation}\label{definition of alpha}        
\alpha:=
\begin{cases}
\min\left(\gamma,\lambda + 1 -\frac{2n}{p}\right),\quad \frac{2n}{\lambda+1} < p \leq 2;\\
\\
\min\left(\gamma,\dfrac{\lambda + 1-n}{p-1}\right),\quad 2 \leq p<n.
\end{cases}
    \end{equation}   
\end{theorem}

\begin{proof}
Aiming to use \Cref{campanato}, we take any ball $B\subset\Omega$ with radius $r>0$, such that $B^\varepsilon\subset\Omega$, where $B^\varepsilon$ is the concentric ball with radius $r+\varepsilon$. Taking the solution $v$ of
\begin{equation}\label{Dirichlet}
	\begin{cases}
		-\Div(|Dv|^{p-2}Dv)=0 \quad \textrm{in} \ B^\varepsilon\\
		v-u\in W^{1,p}_0(B^\varepsilon),
	\end{cases}
\end{equation}    
we estimate   
\begin{equation}\label{splitting integral}
        \begin{split}
            \fint_{B}|Du-(Du)_r|^p\,dx
            &
            \le
            C
            \left(\fint_{B}|Du-Dv|^p\,dx
            +
            \fint_{B}|Dv-(Dv)_r|^p\,dx
            \right.
            \\
            &
            \quad \left.
            +\fint_{B}|(Dv)_r-(Du)_r|^p\,dx
            \right)
            \\            
            &
            =:C(I_1+I_2+I_3),
        \end{split}
\end{equation}
for a constant $C>0$, depending only on $p$. To estimate $I_2$, observe that by \Cref{regularity of v}, $v\in C^{1,\gamma}_{\loc}(B^\varepsilon)$, and the Campanato characterization of H\"older spaces yields
\begin{equation}\label{I2 esimate}
        I_2=\fint_{B}|Dv-(Dv)_r|^p\,dx\le Cr^{\gamma p}.
\end{equation}
Furthermore, using H\"older's inequality, we bound
\begin{equation}\label{I3 estimate}
        I_3=\fint_{B}|(Dv)_r-(Du)_r|^p\,dx\le \fint_{B}|Du-Dv|^p\,dx= I_1.
\end{equation}
To estimate $I_1$, we take $u-v\in W^{1,p}_0(B^\varepsilon)$ as a test function both in \eqref{mainequation} and in \eqref{Dirichlet}, to obtain
\begin{equation*}
        \int_{B^\varepsilon}|Du|^{p-2}Du\cdot D(u-v)\,dx=\int_{B^\varepsilon}f(u-v)\,dx
\end{equation*}
and
$$
    \int_{B^\varepsilon}|Dv|^{p-2}Dv\cdot D(u-v)\,dx=0.
$$
Therefore,
\begin{equation}\label{important iquality}
        \int_{B^\varepsilon}\left(|Du|^{p-2}Du-|Dv|^{p-2}Dv\right)\cdot D(u-v)\,dx=\int_{B^\varepsilon}f(u-v)\,dx.
\end{equation}

\medskip

We now discriminate between the degenerate and singular cases.

\bigskip

\noindent \textsc{Case 1.} For $2\le p<n$, we use the well-known inequality,
\begin{equation*}\label{well-known inequality}
        \left(|\xi|^{p-2}\xi-|\eta|^{p-2}\eta\right)\cdot(\xi-\eta)\ge C |\xi-\eta|^p, \quad \forall \, \xi, \eta \in \R^n,
\end{equation*}
for a constant $C>0$, depending only on $p$. Combined with \eqref{important iquality}, it gives
\begin{align*}
    \int_{B}|Du-Dv|^p\,dx&\le\int_{B^\varepsilon}|Du-Dv|^p\,dx\\
    &\le\frac{1}{C}\int_{B^\varepsilon}\left(|Du|^{p-2}Du-|Dv|^{p-2}Dv\right)\cdot D(u-v)\,dx\\
    &=\frac{1}{C}\int_{B^\varepsilon}f(u-v)\,dx.
\end{align*} 
Letting $\varepsilon\to0$, we get
$$
    \int_{B}|Du-Dv|^p\,dx\le\frac{1}{C}\int_{B}f(u-v)\,dx.
$$    
      
Using now the H\"older and Fefferman--Phong inequalities, we estimate
\begin{eqnarray}\label{applying FPh}
    \int_{B}|f||u-v|\,dx 
    & 
    \le
    & 
    \left(\int_{B} |f|\,dx \right)^{1/p'} \left( \int_{B} |f||u-v|^p\,dx \right)^{1/p}\nonumber\\
    & \le & 
    c\, \|f\|_{L^{1,\lambda}(\Omega)}^{\frac{1}{p'}}r^{\frac{\lambda}{p'}}
    \|f\|_{L^{1,\lambda}(\Omega)}^{\frac{1}{p}}r^{\frac{\lambda+p-n}{p}}\left(\int_{B} |Du-Dv|^p\,dx \right)^{1/p}\nonumber\\
    &=&c\,r^{\lambda+1-\frac{n}{p}}\|f\|_{L^{1,\lambda}(\Omega)}\left(\int_{B}|Du-Dv|^p\,dx\right)^{\frac{1}{p}}.
\end{eqnarray}
Hence, we obtain    
\begin{equation*}
    \left(\int_{B}|Du-Dv|^p\,dx\right)^{\frac{1}{p'}}
    \le C
    r^{\lambda+1-\frac{n}{p}}\|f\|_{L^{1,\lambda}(\Omega)},
\end{equation*}
and thus
\begin{eqnarray}
    I_1 & = & \fint_{B}|Du-Dv|^p\,dx \nonumber\\
    & \le & C r^{(\lambda+1-\frac{n}{p})p '-n}\|f\|^{p'}_{L^{1,\lambda}(\Omega)} \nonumber\\
    & = & C r^{\frac{(\lambda +1 -n)p}{p-1}} \|f\|_{L^{1,\lambda}(\Omega)}^{\frac{p}{p-1}} \label{final-estimate}.
\end{eqnarray}   

Combining \eqref{splitting integral}, \eqref{I2 esimate}, \eqref{I3 estimate} and \eqref{final-estimate}, we reach
\begin{eqnarray*}
        \fint_{B}|Du-(Du)_r|^p\,dx & \le & C(I_1+I_2)\\
        & \le & C\left(r^{\frac{(\lambda + 1-n)p}{p-1}}+r^{\gamma p}\right)\\
        & \leq & C r^{\alpha p},
\end{eqnarray*}
where $C>0$ depends only on $p$, $n$ and $\|f\|_{L^{1,\lambda}(\Omega)}$, and $\alpha$ is defined by \eqref{definition of alpha}. Recalling \Cref{campanato}, we conclude that $u\in C^{1,\alpha}_{\loc}(\Omega)$.

\bigskip

\noindent \textsc{Case 2.} For $\frac{2n}{1+\lambda}<p<2$, using H\"older's inequality with $2/p$ and $2/(2-p)$, we estimate
\begin{align}\label{3.14}
        &\int_B|Du-Dv|^p\,dx\le\int_{B^\varepsilon}|Du-Dv|^p\,dx\nonumber\\
        &=\int_{B^\varepsilon}\left[\left(|Du|+|Dv|\right)^{p-2}|Du-Dv|^2\right]^\frac{p}{2}\left(|Du|+|Dv|\right)^{(2-p)\frac{p}{2}}\,dx\\ &\le\left[\int_{B^\varepsilon}\left(|Du|+|Dv|\right)^{p-2}|Du-Dv|^2\,dx\right]^\frac{p}{2}\left[\int_{B^\varepsilon}\left(|Du|+|Dv|\right)^p\,dx\right]^{\frac{2-p}{2}}.\nonumber
\end{align}
To bound the first term on the right-hand side of \eqref{3.14}, we use the well-known inequality
$$
    \left(|\xi|^{p-2}\xi-|\eta|^{p-2}\eta\right)\cdot(\xi-\eta)\ge c_p\frac{|\xi-\eta|^2}{\left(|\xi|+|\eta|\right)^{2-p}},\quad \forall \xi,\eta\in\R^n,
$$
where $c_p>0$ is a constant depending only on $p$. Additionally, taking $u-v$ as a test function for \eqref{Dirichlet} and using H\"older's inequality, we get
$$
    \int_{B^\varepsilon}|Dv|^p\,dx \leq \int_{B^\varepsilon}|Du|^p\,dx.
$$
Thus, letting $\varepsilon\to0$, \eqref{3.14} leads to
\begin{equation}\label{final estimate}
    \int_B|Du-Dv|^p\,dx\le C\left[\int_Bf(u-v)\,dx\right]^{\frac{p}{2}},
\end{equation}
where $C>0$ depends only on $p$ and $\|Du\|_{L^p(B)}$. Combining \eqref{final estimate} with \eqref{applying FPh}, we arrive at
$$
    \fint_B|Du-Dv|^p\,dx\le C\,r^{(\lambda+1-\frac{n}{p})p-n}\|f\|^{p}_{L^{1,\lambda}(\Omega)}.
$$
Note that since $p>\frac{2n}{1+\lambda}$, then $(\lambda+1-\frac{n}{p})p-n>0$. As before,
\begin{eqnarray*}
        \fint_{B}|Du-(Du)_r|^p\,dx & \le & C(I_1+I_2)\\
        & \le & C\left(r^{(\lambda+1-\frac{n}{p})p-n}+r^{\gamma p}\right)\\
        & \leq & C r^{\alpha p}.
\end{eqnarray*}    
\end{proof}

\begin{remark}
    Observe that, for $p\geq 2$, we always have 
    $$\dfrac{\lambda + 1-n}{p-1} \in (0,1),$$
    while, for $\frac{2n}{\lambda+1}<p \leq 2$, we have
    $$
    \lambda + 1 -\frac{2n}{p}\in(0,1).
    $$
\end{remark}

\begin{remark}\label{remark on optimality}
The Serrin's functions (cf. \cite{Se})
$$
u(x)=|x|^\gamma,\quad 0 < \gamma \leq 1,
$$
which are not more regular than Lipschitz continuous, solve \eqref{mainequation} in $B_1$, with
$$
f=\gamma^{p-1}\left[(\gamma-1)(p-1)-1+n \right]|x|^{(\gamma-1)(p-1)-1}.
$$
Since $f\in L^{1,\lambda} (B_1)$, for $\lambda = (\gamma-1)(p-1)-1+n \leq n-1$, we conclude that our assumption $\lambda > n-1$ cannot be relaxed.
\end{remark}

\section{A class of quasilinear equations}\label{generalization}

As commented earlier, the sharp $C^{1,\alpha}_{\loc}$ regularity result of the previous section holds for a more general class of quasilinear operators. In this section, we state the general regularity result and sketch its proof. Let
$$
A(x,\mu,\xi):\Omega\times\R\times\R^n\rightarrow\R^n
$$
be differentiable and satisfy the structural assumptions
\begin{equation}\label{structural conditions}
    \begin{cases}
        \displaystyle\sum_{i,j=1}^n\frac{\partial A_j}{\partial\xi_i}(x,\mu,\xi)\cdot\zeta_i\zeta_j\ge C_1(\kappa+|\xi|)^{p-2}|\zeta|^2\\
        \displaystyle\sum_{i,j=1}^n\left|\frac{\partial A_j}{\partial\xi_i}(x,\mu,\xi)\right|\le C_2(\kappa+|\xi|)^{p-2}\\
        \displaystyle\sum_{i,j=1}^n\left|\frac{\partial A_j}{\partial x_i}(x,\mu,\xi)\right|+\left|\frac{\partial A_j}{\partial\mu}(x,\mu,\xi)\right|\le C_3(\kappa+|\xi|)^{p-2}|\xi|,
    \end{cases}
\end{equation}
for any $\zeta\in\R^n$ and for some $\kappa\in[0,1]$, for positive constants $C_1$, $C_2$, $C_3$. The right-hand side in \eqref{general equation} is assumed to satisfy 
\begin{equation}\label{assumptions on B}
    |B(x,\mu,\xi)|\le h(x)|\xi|^{p-1}+g(x)|\mu|^{p-1}+f(x),
\end{equation}
where 
\begin{equation}\label{rhs is in Morrey}
    h\in L^{p,\lambda}(\Omega),\quad g,f\in L^{1,\lambda}(\Omega),\quad n-1<\lambda<n.
\end{equation}
Solutions to \eqref{general equation} are understood in the weak sense according to the following definition.

\begin{definition}
    A function $u\in W^{1,p}(\Omega)$ is called a weak solution to \eqref{general equation} if
    $$
    \int_\Omega A(x,u,Du)\cdot D\varphi\,dx=\int_\Omega B(x,u,Du)\varphi\,dx,
    \quad \forall\varphi\in W^{1,p}_0(\Omega).
    $$
\end{definition}

As before, the definition is justified by the Fefferman--Phong inequality. 
Furthermore, we recall an analog of \Cref{regularity of v} (see \cite{LU68} and \cite{DiB83,T84}, for example).

\begin{lemma}\label{general lemma}
If $p>1$ and $u\in W^{1,p}(\Omega)$, then the problem
\begin{equation}\label{homogeneous}
	\begin{cases}
		-\Div A(x,u,Du)=0 \quad \mbox{in} \ \Omega\\
		v-u\in W^{1,p}_0(\Omega)
	\end{cases}
\end{equation}
has a unique solution $v\in C_{\loc}^{1,\gamma}(\Omega)$, for a certain $\gamma\in(0,1)$.
\end{lemma}

As in \Cref{results_section}, \Cref{general lemma} leads to the following regularity result.

\begin{theorem}\label{general result} 
Let \eqref{structural conditions}-\eqref{rhs is in Morrey} hold, and $\gamma\in(0,1)$ be as in \Cref{general lemma}. If $u$ is a weak solution of \eqref{general equation}, then $u\in C_{\loc}^{1,\alpha}(\Omega)$, where $\alpha$ is defined by \eqref{definition of alpha}. 
\end{theorem}

\begin{proof}[Sketch of the proof.] For $2\le p<n$, 
we have (cf. \cite[Lemma 1]{T84})  
$$
\left(A(x,\mu,\xi)-A(x,\mu,\eta)\right)\cdot(\xi-\eta)\ge C |\xi-\eta|^p,
$$
for a constant $C>0$, depending on $p$, $n$ and the $C_i$ in \eqref{structural conditions}. Combined with \eqref{assumptions on B}, as in the proof of \Cref{main result}, this leads to \eqref{final-estimate}, where $v\in C_{\loc}^{1,\gamma}(B^\varepsilon)$ is the unique solution of \eqref{homogeneous} in $B^\varepsilon$. The rest of the proof follows analogously.

If $\frac{2n}{\lambda+1}<p<2$, we use
$$
    \left(A(x,\mu,\xi)-A(x,\mu,\eta)\right)\cdot(\xi-\eta)\ge c\frac{|\xi-\eta|^2}{\left(\kappa+|\xi|+|\eta|\right)^{2-p}},\quad \forall \xi,\eta\in\R^n,
    $$
where $c>0$ is a constant depending only on $p$ and the structural constants. The rest of the proof proceeds identically.
\end{proof}

\bigskip

{\small \noindent{\bf Acknowledgments.} 
The authors thank the anonymous referee for the insightful comments and constructive suggestions that substantially improved the manuscript. This publication is based upon work supported by King Abdullah University of Science and Technology (KAUST) under Award No. ORFS-CRG12-2024-6430. GDF is part of Gruppo Nazionale per l'Analisi Matematica, la Probabilit\`a e le loro Applicazioni (GNAMPA) - Istituto Nazionale di Alta Matematica (INdAM) and partially supported by the University of Catania, Piano della Ricerca 2016/2018 Linea di intervento 2. JMU is partially supported by UID/00324 - Centre for Mathematics of the University of Coimbra. GDF thanks KAUST for the nice hospitality and excellent research environment during his visits.}

\end{document}